\documentclass{article}

\usepackage{mystyle}

\title{Induced 2-degenerate Subgraphs of Triangle-free Planar Graphs}
\author{
  Zden\v{e}k Dvo\v{r}\'ak%
  \thanks{Charles University, Prague, Czech Republic,
  {\tt rakdver@iuuk.mff.cuni.cz}.  Supported by project 17-04611S (Ramsey-like aspects of graph coloring) of Czech Science Foundation.}
  \and
  Tom Kelly%
  \thanks{Department of Combinatorics and Optimization, University of Waterloo, 200 University Ave West, Waterloo, Ontario, Canada N2L 3G1, \texttt{t9kelly@uwaterloo.ca}}
}
\date{February 20, 2018}

\begin{document}
\maketitle

\begin{abstract}
  A graph is \textit{$k$-degenerate} if every subgraph has minimum degree at most $k$.  We provide lower bounds on the size of a maximum induced 2-degenerate subgraph in a triangle-free planar graph.  We denote the size of a maximum induced 2-degenerate subgraph of a graph $G$ by $\degen_2(G)$.  We prove that if $G$ is a connected triangle-free planar graph with $n$ vertices and $m$ edges, then $\degen_2(G) \geq \frac{6n - m - 1}{5}$.  By Euler's Formula, this implies $\degen_2(G) \geq \frac{4}{5}n$.  We also prove that if $G$ is a triangle-free planar graph on $n$ vertices with at most $n_3$ vertices of degree at most three, then $\degen_2(G) \geq \frac{7}{8}n - \bigOconstant n_3$.  
\end{abstract}

\section{Introduction}

A graph is \textit{$k$}-degenerate if every nonempty subgraph has a vertex of degree at most $k$.  The \textit{degeneracy} of a graph is the smallest $k$ for which it is $k$-degenerate, and it is one less than the \textit{coloring number}.  It is well-known that planar graphs are 5-degenerate and that triangle-free planar graphs are 3-degenerate.  The problem of bounding the size of an induced subgraph of smaller degeneracy has attracted a lot of attention.  In this paper we are interested in lower bounding the size of maximum induced 2-degenerate subgraphs in triangle-free planar graphs.  In particular, we conjecture the following.

\begin{conj}\label{main-conj}
Every triangle-free planar graph contains an induced 2-degenerate subgraph on at least $\frac{7}{8}$ of its vertices.
\end{conj}

Conjecture~\ref{main-conj}, if true, would be tight for the \textit{cube}, which is the unique 3-regular triangle-free planar graph on 8 vertices (see Figure~\ref{cube figure}).
For an infinite class of tight graphs, if $G$ is a planar triangle-free graph whose vertex set can be partitioned into parts each inducing a subgraph isomorphic to the cube, then
$G$ does not contain an induced 2-degenerate subgraph on more than $\tfrac{7}{8}|V(G)|$ vertices.

\begin{figure}
  \centering
  \begin{tikzpicture}
  \tikzstyle{every node}=[fill, draw, circle, scale=.5, font=\huge];
  \node[label=45:$u_1$] (u1) at (0, 0) {};
  \node[label=-45:$u_2$] (u2) at (1, 0) {};
  \node[label=-135:$u_3$] (u3) at (1, -1) {};
  \node[label=135:$u_4$] (u4) at (0, -1) {};
  \draw (u1) -- (u2) -- (u3) -- (u4) -- (u1);
  \node[label=-45+180:$v_1$] (v1) at ($(u1) + (135:1)$) {};
  \draw (u1) -- (v1);
  \node[label=-135+180:$v_2$] (v2) at ($(u2) + (-135+180:1)$) {};
  \draw (u2) -- (v2);
  \node[label=135+180:$v_3$] (v3) at ($(u3) + (135+180:1)$) {};
  \draw (u3) -- (v3);
  \node[label=45+180:$v_4$] (v4) at ($(u4) + (45+180:1)$) {};
  \draw (u4) -- (v4);
  \draw (v1) -- (v2) -- (v3) -- (v4) -- (v1);
\end{tikzpicture}
  \caption{The cube.}
  \label{cube figure}
\end{figure}
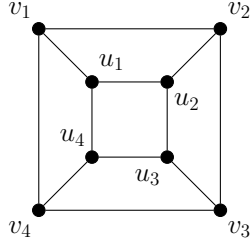

Towards Conjecture~\ref{main-conj}, we prove the following weaker bound.

\begin{thm}\label{four-fifths-thm}
  Every triangle-free planar graph contains an induced 2-degenerate subgraph on at least $\frac{4}{5}$ of its vertices.
\end{thm}

We believe the argument we use can be strengthened to give a bound $\frac{5}{6}$, however the technical issues are substantial
and since we do not see this as a viable way to prove Conjecture~\ref{main-conj} in full, we prefer to present the easier argument
giving the bound $\frac{4}{5}$.

Triangle-free planar graphs have average degree less than $4$, and thus they must contain some vertices of degree at most three.
Nevertheless, they may contain only a small number of such vertices---there exist arbitrarily large triangle-free planar graphs
of minimum degree three that contain only $8$ vertices of degree three.  It is natural to believe that $2$-degenerate induced subgraphs
are harder to find in graphs with larger vertex degrees, and thus one might wonder whether a counterexample to Conjecture~\ref{main-conj}
could not be found among planar triangle-free graphs with almost all vertices of degree at least four.  This is a false intuition---such
graphs are very close to being $4$-regular grids, and their regular structure makes it possible to find large $2$-degenerate induced subgraphs.
To support this counterargument, we prove the following approximate form of Conjecture~\ref{main-conj} for graphs with small numbers
of vertices of degree at most three.

\begin{thm}\label{bounded-degree-three}
  If $G$ is a triangle-free planar graph on $n$ vertices with $n_3$ vertices of degree at most three, then $G$ contains an induced 2-degenerate subgraph on at least $\frac{7}{8}n - \bigOconstant n_3$ vertices.
\end{thm}

Theorems~\ref{four-fifths-thm} and \ref{bounded-degree-three} are corollaries of more technical results.

\begin{define}
  We say a graph is \textit{difficult} if it is connected, every block is either a vertex, an edge, or isomorphic to the cube, and any two blocks isomorphic to the cube are vertex-disjoint.
\end{define}

We actually prove the following, which easily implies Theorem~\ref{four-fifths-thm} since, by Euler's formula, a triangle-free planar graph $G$ on at least three vertices satisfies $|E(G)| \leq 2|V(G)| - 4$.

\begin{thm}\label{4/5 bound}
  If $G$ is a triangle-free planar graph on $n$ vertices with $m$ edges and $\lambda$ difficult components, then $G$ contains an induced 2-degenerate subgraph on at least
  \begin{equation*}
    \frac{6n - m - \lambda}{5}
  \end{equation*}
  vertices.
\end{thm}

The proof of Theorem~\ref{4/5 bound} is the subject of Section~\ref{four-fifths-section}.

\begin{define}
  If $G$ is a plane graph, we let $\threefaces(G)$ denote the minimum size of a set of faces such that every vertex in $G$ of degree at most three is incident to at least one of them.
\end{define}

We actually prove the following, which easily implies Theorem~\ref{bounded-degree-three}.

\begin{thm}\label{real-bigO-thm}
  If $G$ is a triangle-free plane graph on $n$ vertices, then either $G$ is 2-degenerate or $G$ contains an induced 2-degenerate subgraph on at least
  \begin{equation*}
    \frac{7}{8}n - \bigOconstant \left(\threefaces(G) - 2\right)
  \end{equation*}
  vertices. 
\end{thm}

The proof of Theorem~\ref{real-bigO-thm} is the subject of Section~\ref{bounded-degree-three-section}.

Let us discuss some related results.  To simplify notation, for a graph $G$ we let $\degen_k(G)$ denote the size of a maximum induced subgraph that is $k$-degenerate.  Alon, Kahn, and Seymour \cite{AKS87} proved in 1987 a general bound on $\degen_k(G)$ based on the degree sequence of $G$.  They derive as a corollary that if $G$ is a graph on $n$ vertices of average degree $d\geq 2k$, then $\degen_k(G) \geq \frac{k + 1}{d + 1}n$.  Since triangle-free planar graphs have average degree at most four, this implies that if $G$ is triangle-free and planar then $\degen_2(G) \geq \frac{3}{5}n$.  Our Theorem~\ref{four-fifths-thm} improves upon this bound.

For the remainder of this section, let $G$ be a planar graph on $n$ vertices.

Note that a graph is 0-degenerate if and only if it is an independent set.  The famous Four Color Theorem, the first proof of which was announced by Appel and Haken \cite{AH76} in 1976, implies that $\degen_0(G) \geq \frac{1}{4}n$.  In the same year,
Albertson \cite{A76} proved the weaker result that $\degen_0(G) \geq \frac{2}{9}n$, which was improved to $\degen_0(G) \geq \frac{3}{13}n$ by Cranston and Rabern~\cite{cranston2016planar};
the constant factor $\frac{3}{13}$ is the best known to date without using the Four Color Theorem.  The factor $\frac{1}{4}$ is easily seen to be best possible by considering copies of $K_4$.

If additionally $G$ is triangle-free, a classical theorem of Gr\H otzsch \cite{G58} says that $G$ is 3-colorable, and therefore $\degen_0(G) \geq \frac{n}{3}$.  In fact, Steinberg and Tovey \cite{ST93} proved that $\degen_0(G) \geq \frac{n + 1}{3}$, and a construction of Jones \cite{J84} implies this is best possible.  Dvo\u{r}\'ak and Mnich \cite{DM17} proved that there exists $\varepsilon > 0$ such that if $G$ has girth at least five, then $\degen_0(G) \geq \frac{n}{3 - \varepsilon}$.

Note that a graph is 1-degenerate if and only if it contains no cycles.  In 1979, Albertson and Berman \cite{AB79} conjectured that every planar graph contains an induced forest on at least half of its vertices, i.e.\ $\degen_1(G) \geq \frac{1}{2}n$.  The best known bound for $\degen_1(G)$ for planar graphs is $\frac{2}{5}n$, which follows from a classic result of Borodin \cite{B76} that planar graphs are acyclically 5-colorable.

Akiyama and Watanabe \cite{AW87} conjectured in 1987 that if additionally $G$ is bipartite then $\degen_1(G) \geq \frac{5n}{8}$, and this may also be true if $G$ is only triangle-free.  The best known bound when $G$ is bipartite is $\degen_1(G) \geq \lceil\frac{4n + 3}{7}\rceil$, which was proved by Wan, Xie, and Yu \cite{WXY17}.  The best known bound when $G$ is triangle-free is $\degen_1(G) \geq \frac{5}{9}n$, which was proved by Le \cite{L16} in 2016.  Kelly and Liu \cite{KL17} proved that if $G$ has girth at least five, then $\degen_1(G) \geq \frac{2}{3}n$.

Kierstead, Oum, Qi, and Zhu \cite{OZ+} proved that if $G$ is a planar graph on $n$ vertices then $\degen_3(G) \geq \frac{5}{7}n$, but the proof is yet to appear.  A bound for $\degen_3(G)$ of $\frac{5}{6}n$ may be possible, which is achieved by both the octahedron and the icosahedron.

In 2015, Lukot'ka, Maz\' ak, and Zhu \cite{LMZ15} studied $\degen_4$ for planar graphs.  They proved that $\degen_4(G) \geq \frac{8}{9}n$.  A bound for $\degen_4(G)$ of $\frac{11}{12}n$ may be possible, which is achieved by the icosahedron.

So far, bounds on $\degen_2(G)$ for planar graphs have not been studied.  However, as Lukot'ka, Maz\' ak, and Zhu \cite{LMZ15} pointed out, it is easy to see that every planar graph contains an induced outerplanar subgraph on at least half of its vertices.  Since outerplanar graphs are 2-degenerate, this implies $\degen_2(G) \geq \frac{1}{2}n$.  Nevertheless, a bound of $\degen_2(G) \geq \frac{2}{3}n$ may be possible, which is achieved by the octahedron.  If $G$ has girth at least five, $\degen_2(G)$ may be as large as $\frac{19}{20}n$, which is achieved by the dodecahedron.


\section{Proof of Theorem \ref{4/5 bound}}\label{four-fifths-section}
In this section we prove Theorem \ref{4/5 bound}.  First we prove some properties of a hypothetical minimal counterexample (i.e.,
a plane triangle-free graph $G$ with the smallest number $n$ of vertices such that $\alpha_2(G)<\frac{6n - m - \lambda}{5}$,
where $m=|E(G)|$ and $\lambda$ is the number of difficult components of $G$).

\subsection{Preliminaries}

\begin{lemma}\label{no difficult components}
  A minimal counterexample $G$ to Theorem \ref{4/5 bound} is connected and has no difficult components.
\end{lemma}
\begin{proof}
  Note that the union of induced 2-degenerate subgraphs from each component of
  $G$ is an induced 2-degenerate subgraph of $G$.  Thus if $G$ is not
  connected, then one of its components is a smaller counterexample, a
  contradiction.
  
  Now suppose for a contradiction that $G$ has a difficult component.  Since $G$ is connected, $G$ is difficult.
  Note that $G$ is not a single vertex, or else $G$ is not a counterexample.
  Suppose that $G$ contains a vertex $x$ of degree $1$; in this case, note that $G-x$ is a difficult graph.
  Since $G$ is a minimal counterexample, there exists a set $S\subseteq V(G-x)$ that induces a 2-degenerate subgraph of size at least
  \begin{equation*}
    \frac{6|V(G-x)| - |E(G-x)| - 1}{5}=\frac{6|V(G)| - |E(G)| - 1}{5} - 1.
  \end{equation*}
  But then $S\cup \{x\}$ induces a 2-degenerate subgraph in $G$, contradicting that $G$ is a minimal counterexample.
  
  Therefore, $G$ has minimum degree at least $2$.  Note that $G$ is not a cube, or else $G$ is not a counterexample.
  Since $G$ is difficult, we conclude that $G$ is not $2$-connected and any end-block of $G$ is a cube.
  Let $X$ be the vertex set of an end-block of $G$, and observe that $G-X$ is a difficult graph.
  Since $G$ is a minimal counterexample, there exists a set $S\subseteq V(G-X)$ that induces a 2-degenerate subgraph of size at least
  \begin{equation*}
    \frac{6|V(G-X)| - |E(G-X)| - 1}{5}=\frac{6|V(G)| - |E(G)| - 1}{5} - 7.
  \end{equation*}
  But then for any $v\in X$, $S\cup X\setminus\{v\}$ induces a 2-degenerate subgraph in $G$, contradicting that $G$ is a minimal counterexample.
\end{proof}

We will often make use of the following induction lemma.
\begin{lemma}\label{induction lemma}
  Let $G$ be a minimal counterexample to Theorem \ref{4/5 bound}, and let $X\subseteq V(G)$.  If every induced 2-degenerate subgraph of $G-X$ can be extended to one of $G$ by adding $A$ vertices, then
  \begin{equation*}
    \lambda' \ge 5A - 6|X| + |E(G)| - |E(G-X)| + 1,
  \end{equation*}
  where $\lambda'$ is the number of difficult components of $G-X$.
\end{lemma}
\begin{proof}
  Let $S\subseteq V(G-X)$ induce a maximum 2-degenerate subgraph in $G-X$.  Since $G$ is a minimal counterexample, 
  \begin{equation}\label{lower bound for G-X}
    |S| \geq \frac{6(|V(G)| - |X|) - |E(G-X)| - \lambda'}{5}.
  \end{equation}
  Note that $G$ has no difficult components by Lemma \ref{no difficult components}.  Since $S$ can be extended to induce a 2-degenerate subgraph in $G$ by adding $A$ vertices of $X$,
  \begin{equation}\label{upper bound for G-X}
    |S| + A < \frac{6|V(G)| - |E(G)|}{5}.
  \end{equation}
  Combining \eqref{lower bound for G-X} and \eqref{upper bound for G-X} yields $\lambda' > 5A - 6|X| + |E(G)| - |E(G-X)|$,
  which gives the desired inequality since both sides are integers.
\end{proof}

\begin{lemma}\label{edges leaving cube}
  A minimal counterexample $G$ to Theorem \ref{4/5 bound} has no subgraph isomorphic to the cube that has fewer than six edges leaving.
\end{lemma}
\begin{proof}
  Let $X = \{v_1, v_2, v_3, v_4, u_1, u_2, u_3, u_4\}$ induce a cube in $G$ where $v_1v_2v_3v_4v_1$ and $u_1u_2u_3u_4u_1$ are 4-cycles and $v_i$ is adjacent to $u_i$ for each $i\in\{1, 2, 3, 4\}$, as in Figure \ref{cube figure}.
  Suppose for a contradiction that $|E(X, V(G-X))| \leq 5$.  Let $S$ induce a 2-degenerate subgraph in $G - X$.
  First, we claim that there is some vertex $v\in X$ such that $S\cup X\setminus\{v\}$ induces a 2-degenerate subgraph in $G$.

  If $v_1$ has at least three neighbors not in $X$, then $S\cup X\setminus\{v_1\}$ induces a 2-degenerate subgraph in $G$:
  Since $G[S]$ is $2$-degenerate, it suffices to verify that for every non-empty $X'\subseteq X\setminus\{v_1\}$, there
  exists a vertex $x\in X'$ with at most two neighbors in $S\cup X'$.  Since the cube is $3$-edge-connected, there are
  at least three edges with one end in $X'$ and the other end in $X\setminus X'$.  Since there are at most five edges leaving $X$
  and at least three of them are incident with $v_1$, at most two such edges are incident with vertices of $X'$.
  Consequently, $\sum_{x\in X'} \deg_{G[X'\cup S]} (x)\le 3|X'|-3+2$, and thus $X'$ indeed contains a vertex whose degree in $G[X'\cup S]$ is
  less than three.

  By symmetry, we may assume no vertex in $X$ has more than two neighbors not in $X$.
  If $v_1$ has two neighbors not in $X$, an analogous argument using the fact that the only $3$-edge-cuts in the cube are the neighborhoods
  of vertices shows that $S\cup X\setminus\{v_1\}$ induces a 2-degenerate subgraph in $G$, unless each of $u_1$, $v_2$, and $v_4$ has
  a neighbor not in $X$.  However, in that case it is easy to verify that $S\cup X\setminus\{u_1\}$ induces a 2-degenerate subgraph in $G$.

  Hence, we may assume that each vertex of $X$ has at most one neighbor not in $X$.  Let $Z\subseteq X$ be a set of size exactly $5$
  containing all vertices of $X$ with a neighbor outside of $X$.  If $Z$ contains all vertices of a face of the cube, then by symmetry
  we can assume that $Z=\{v_1,v_2,v_3,v_4,u_1\}$, and $S\cup X\setminus\{v_2\}$ induces a 2-degenerate subgraph in $G$.
  Otherwise, we have $|Z\cap \{v_1,v_2,v_3,v_4\}|\le 3$ and $|Z\cap \{u_1,u_2,u_3,u_4\}|\le 3$, and since $|Z|=5$, by symmetry we can assume that
  $|Z\cap \{v_1,v_2,v_3,v_4\}|=2$ and $v_1\in Z$.  However, then $S\cup X\setminus\{v_1\}$ induces a 2-degenerate subgraph in $G$.

  This confirms that every set inducing a $2$-degenerate subgraph of $G-X$ can be extended to a set inducing a $2$-degenerate subgraph of $G$
  by the addition of $7$ vertices.
  Let $\lambda'$ be the number of difficult components of $G-X$.  By Lemma \ref{induction lemma}, $\lambda' \geq |E(X, V(G-X))|$.
  Since $G$ is connected, it follows that $G-X$ consists of exactly $|E(X, V(G-X))|$ difficult components, each connected by exactly
  one edge to the cube induced by $X$.  But then $G$ is a difficult graph, contradicting Lemma \ref{no difficult components}.
\end{proof}

\begin{lemma}\label{min degree three}
  A minimal counterexample $G$ to Theorem \ref{4/5 bound} has minimum degree at least three.
\end{lemma}
\begin{proof}
  Suppose not.  Let $v\in V(G)$ be a vertex of degree at most two.  Note that $v$ has degree at least one by Lemma \ref{no difficult components}.  Note also that any induced 2-degenerate subgraph of $G-v$ can be extended  to one of $G$ by adding $v$.  By Lemma \ref{induction lemma}, if $G-v$ has $\lambda'$ difficult components, then $\lambda' \geq \deg(v)$.  But then $G$ is a difficult graph, contradicting Lemma \ref{no difficult components}.
\end{proof}


\subsection{Reducing vertices of degree three}

A cycle $C$ in a plane graph is \emph{separating} if both the interior and the exterior of $C$ contain at least one vertex.
The main result of this subsection is the following lemma.
\begin{lemma}\label{no special vertices}
  A minimal counterexample $G$ to Theorem \ref{4/5 bound} contains no vertex of degree three that is not contained in a separating cycle of length four or five.
\end{lemma}

For the remainder of this subsection, let $G$ be a minimal counterexample to Theorem \ref{4/5 bound}, and suppose $v\in V(G)$ is a vertex of degree three that is not contained in a separating cycle of length four or five.  Recall that a minimal counterexample is a plane graph, so $G$ has a fixed embedding.

\begin{claim}\label{no adjacent 5-vertex}
 The vertex $v$ has no neighbors of degree at least five.
\end{claim}
\begin{proof}
  Suppose for a contradiction $v$ has a neighbor $u$ of degree at least five, and let $X = \{u, v\}$.  Note that any induced 2-degenerate subgraph of $G-X$ can be extended to one of $G$ including $v$.  By Lemma \ref{induction lemma}, the number of difficult components of $G-X$ is positive.

  Let $D$ be a difficult component of $G-X$.
  First, suppose $D$ contains a vertex of degree at most one.  By Lemma \ref{min degree three}, this vertex is adjacent to $u$ and $v$, contradicting that $G$ is triangle-free.
  Therefore $D$ has an end-block $B$ isomorphic to the cube.  Since $G$ is triangle-free and planar, $u$ has at most two neighbors in $B$, and $u$ and $v$ do not both have two neighbors in $B$.  Hence $|E(X, V(B))| \leq 3$, so $B$ has at most four edges leaving, contradicting Lemma \ref{edges leaving cube}.  
\end{proof}

\begin{claim}\label{no adjacent 3-vertex}
  The vertex $v$ has no neighbors of degree three.
\end{claim}
\begin{proof}
  Let $u_1, u_2,$ and $u_3$ be the neighbors of $v$, and suppose for a contradiction that $u_1$ has degree three.
  
  First, let us consider the case $u_2$ has degree at least four
  (and thus exactly four by Claim~\ref{no adjacent 5-vertex}).  Note that any induced 2-degenerate subgraph of $G-\{u_1, u_2, v\}$ can be extended to one of $G$ including $v$ and $u_1$.  By Lemma \ref{induction lemma}, the number of difficult components of $G-\{u_1, u_2, v\}$ is positive.

  Let $D$ be a difficult component of $G-\{u_1, u_2, v\}$.  Note that each leaf of $D$ is adjacent to $u_1$ and $u_2$ and not adjacent to $v$ by Lemma \ref{min degree three}, since $G$ is triangle-free.  Now if $D$ has at least two leaves, then $v$ is contained in a separating cycle of length four, a contradiction.  Note also that $D$ is not an isolated vertex.  Hence, $D$ contains an end-block $B$ isomorphic to the cube.  If $D$ contains another
end-block, then we can choose $B$ among the end-blocks isomorphic to the cube so that $B$ has at most five edges leaving, contradicting Lemma \ref{edges leaving cube}.  Therefore $D$ is isomorphic to the cube.  By Lemma \ref{edges leaving cube}, every neighbor of $u_1, u_2,$ and $v$ is in $D$, contradicting that $G$ is planar and triangle-free.

  Therefore we may assume $u_2$ and symmetrically $u_3$ have degree three.  Note that any induced 2-degenerate subgraph of $G - \{u_1, u_2, u_3, v\}$ can be extended to one of $G$ including $u_1, u_2,$ and $u_3$.  By Lemma \ref{induction lemma}, the number of difficult components of $G-\{u_1, u_2, u_3, v\}$ is positive.

  Let $D$ be a difficult component of $G-\{u_1, u_2, u_3, v\}$.
  First, suppose $D$ is a tree.  If $D$ is an isolated vertex, this vertex is adjacent to $u_1, u_2,$ and $u_3$ by Lemma \ref{min degree three}, but then $v$ is contained in a separating cycle of length four, a contradiction.  Note that $D$ is not an edge, or else it is contained in a triangle with one of $u_1, u_2,$ or $u_3$, by Lemma \ref{min degree three}.  Similarly, $D$ is not a path, or else $G$ contains a triangle or a vertex of degree at most two.  Therefore $D$ has at least three leaves.  Since $G$ has minimum degree three and $\{u_1, u_2, u_3, v\}$ has only six edges leaving, $D$ is isomorphic to $K_{1, 3}$.  In this case, $G$ is isomorphic to the cube, a contradiction.

  Therefore we may assume $D$ is not a tree, so $D$ contains a block isomorphic to the cube.  Let $B$ be a block in $D$ isomorphic to the cube with the fewest edges leaving.  If $D$ contains an endblock different from $B$, then at most five edges are leaving $B$, contradicting Lemma \ref{edges leaving cube}.  Therefore $D$ is isomorphic to the cube and all six edges leaving $\{u_1, u_2, u_3, v\}$ end in $D$, contradicting that $G$ is planar and triangle-free.
\end{proof}

\begin{claim}\label{two 3-vertices in 4-cycle}
  The vertex $v$ is not contained in a cycle of length four that contains another vertex of degree three.
\end{claim}
\begin{proof}
  Suppose for a contradiction that $u_1$ and $u_2$ are neighbors of $v$ with a common neighbor $w$ of degree three that is distinct from $v$, and let $X = \{u_1, u_2, v, w\}$.  By Claims \ref{no adjacent 5-vertex} and \ref{no adjacent 3-vertex}, $u_1$ and $u_2$ have degree four.  Note that any induced 2-degenerate subgraph of $G - X$ can be extended to one of $G$ including $X\setminus\{u_1\}$.  By Lemma \ref{induction lemma}, if $\lambda'$ is the number of difficult components of $G-X$, then $\lambda' \ge 2$.

  Let $D_1$ and $D_2$ be difficult components of $G-X$.  Since there are only six edges leaving $X$, we may assume without loss of generality that $|E(X, V(D_1))|  \leq 3$.  Note that $D_1$ is not an isolated vertex by Lemma \ref{min degree three} since $G$ is triangle-free.  If $D_1$ contains a leaf, then it is adjacent to either both $u_1$ and $u_2$ or both $v$ and $w$ by Lemma \ref{min degree three}.  In either case, $v$ is contained in a separating cycle of length four, a contradiction.  Therefore $D_1$ contains an end-block isomorphic to the cube, contradicting Lemma \ref{edges leaving cube}.
\end{proof}

\begin{claim}\label{3-vertex in two 4-cycles}
  Every edge incident with $v$ is contained in a cycle of length four.
\end{claim}
\begin{proof}
  Suppose for a contradiction $u$ is a neighbor of $v$ such that the edge $uv$ is not contained in a cycle of length four.  Let $G'$ be the graph obtained from $G$ by contracting the edge $uv$ into a new vertex, say $w$, and observe that $G'$ is planar and triangle-free.

  Let $S\subseteq V(G')$ induce a maximum-size induced 2-degenerate subgraph of $G'$.  We claim that $G$ contains an induced 2-degenerate subgraph on at least $|S| + 1$ vertices.  If $w\notin S$, then $S\cup\{v\}$ induces a 2-degenerate subgraph of $G$ on at least $|S| + 1$ vertices, as claimed.  Therefore we may assume $w\in S$.  It suffices to show $S\setminus\{w\}\cup\{u,v\}$ induces a 2-degenerate subgraph in $G$.  Given $S'\subseteq S\setminus\{w\}\cup\{u,v\}$, we will show $G[S']$ contains a vertex of degree at most two.  If $S'\cap\{u,v\} = \varnothing$, then $G[S']$ equals $G'[S']$, which contains a vertex of degree at most two, as desired.  Therefore we may assume $S'\cap \{u,v\}\neq\varnothing$.  Note that $G'[S'\cup\{w\}\setminus\{u,v\}]$ contains a vertex $x$ of degree at most two.   If $x\neq w$, then since $G$ is triangle-free, $x$ is not adjacent to both $u$ and $v$,
  and thus $x$ has degree at most two in $G[S']$, as desired.  So we may assume $w$ has degree at most two in $G'[S'\cup\{w\}\setminus\{u,v\}]$.  Now at least one of $u$ and $v$ has degree at most two in $G[S']$, as desired.

  Since $G$ is a minimal counterexample and $|V(G')|<|V(G)|$, we have
  \begin{equation*}
    |S| \geq \frac{6|V(G')| - |E(G')| - \lambda'}{5} = \frac{6|V(G)| - |E(G)| - \lambda'}{5} - 1,
  \end{equation*}
  where $\lambda'$ is the number of difficult components of $G'$.
  Furthermore, $G$ contains an induced 2-degenerate subgraph on at least $|S| + 1$ vertices as argued, and thus
  \begin{equation*}
    |S| + 1 < \frac{6|V(G)| - |E(G)|}{5}.
  \end{equation*}
  
  It follows that $\lambda' > 0$.  Since $G'$ is connected, $G'$ is difficult.  By Lemmas~\ref{edges leaving cube} and \ref{min degree three},
  $G'$ cannot have an endblock not containing $w$, and thus $G'$ is isomorpic to the cube.  But then either $u$ or $v$ has degree at most two in $G$,
  which is a contradiction.
\end{proof}

We can now prove Lemma \ref{no special vertices}.
\begin{proof}[Proof of Lemma \ref{no special vertices}]
  Suppose for a contradiction that $G$ contains such a vertex $v$. By Claim \ref{3-vertex in two 4-cycles}, the vertex $v$ has a neighbor $u$ such that the edge $uv$ is contained in two cycles of length four.  Let $x_1$ and $x_2$ denote the other neighbors of $v$.  Since $uv$ is contained in two cycles of length four, for $i\in\{1,2\}$, $u$ and $x_i$ have a common neighbor $y_i$ that is distinct from $v$.  By Claims \ref{no adjacent 5-vertex} and \ref{no adjacent 3-vertex}, $u, x_1,$ and $x_2$ have degree four.  Since $v$ is not contained in a separating cycle of length four, $y_1\neq y_2$, $x_1$ and $y_2$ are not adjacent, and $x_2$ and $y_1$ are not adjacent.  By Claim \ref{3-vertex in two 4-cycles}, $y_1$ and $y_2$ have degree at least four.  Let $X = \{v, u, x_1, x_1, y_1, y_2\}$, and note that $|E(G)| - |E(G - X)| = 8+\deg(y_1)+\deg(y_2)$.  Note also that any induced 2-degenerate subgraph of $G-X$ can be extended to one of $G$ by adding $u, v, x_1,$ and $x_2$.  By Lemma \ref{induction lemma}, if $\lambda'$ is the number of difficult components of $G-X$, $\lambda' \ge \deg(y_1)+\deg(y_2)-7\ge 1$.
  Let $D$ be a difficult component of $G - X$ such that the number of edges between $D$ and $X$ is minimum.  Note that if $\deg(y_1)\geq 5$ or $\deg(y_2)\geq 5$, then $|E(V(D), X)| \leq 5$.  Otherwise, $|E(V(D), X)| \leq 9$.

  Since $G$ is triangle-free and $v$ is not contained in a separating cycle of length at most 5,
each vertex of $D$ has at most two neighbors in $X$, and if it has two, these neighbors are either $\{x_1,x_2\}$ or $\{y_1,y_2\}$.
By Claim~\ref{two 3-vertices in 4-cycle}, if $z$ is a leaf of $D$, we conclude that $z$ is adjacent to $y_1$ and $y_2$.
By planarity, $D$ has at most two leaves.  Furthermore, if $D$ had two leaves, then all edges between $D$ and $X$
would be incident with $y_1$ and $y_2$, and by planarity and absence of triangles, we would conclude that
$G$ contains a vertex of degree two or a cube subgraph with at most four edges leaving, which is a contradiction.
Hence, $D$ has an end-block $B$ isomorphic to the cube.  Label the vertices of $B$ according to Figure~\ref{cube figure}.  By Lemma~\ref{edges leaving cube}, $D$ has at most one end-block isomorphic to the cube.  Hence, either $D = B$, or $D$ has precisely two end-blocks, one of which is a leaf and one of which is $B$.

Suppose $\deg(y_1)\ge 5$ or $\deg(y_2)\ge 5$.  Then there are at most 5 edges between $X$ and $D$.  By Lemma~\ref{edges leaving cube}, $B\neq D$, so $D$ has at least two end-blocks.  Therefore there are at most $3$ edges between $B$ and $X$, so there are at most $4$ edges leaving $B$, contradicting Lemma~\ref{edges leaving cube}.  Hence, $\deg(y_1)=\deg(y_2)=4$.

By planarity, all edges between $B$ and $X$ are contained in
one face of $B$.  Since $G$ is triangle-free and $v$ is not contained in a separating $4$-cycle, there are at
most 3 edges between $B$ and $\{x_1,x_2\}$.
If $D$ has a leaf, then as we observed before, the leaf is adjacent to $y_1$ and $y_2$, and by planarity, all edges between $B$
and $X$ are incident with either $\{y_1,y_2,u\}$ or $\{x_1,x_2,y_1,y_2\}$.  By Lemma~\ref{edges leaving cube}, the former is not
possible, and in the latter case, there are $3$ edges between $B$ and $\{x_1,x_2\}$, both $y_1$ and $y_2$ have a neighbor in $B$,
and $D$ consists of $B$ and the leaf.  However, this is not possible, since $G$ is triangle-free.
Consequently, $D$ is isomorphic to the cube.

Let us now consider the case that $u$ has a neighbor in $V(D)$.  We may assume without loss of generality that $u$ is adjacent to $v_1$.  Since $v$ is not in a separating cycle of length at most five, $x_1$ and $x_2$ are not adjacent to $v_1, v_2,$ or $v_4$.  Therefore $x_1$ and $x_2$ each have at most one neighbor in $V(D)$.  By Lemma~\ref{edges leaving cube}, one of $y_1$ and $y_2$ has two neighbors in $V(D)$, and we may assume without loss of generality it is $y_1$.  Since $G$ is planar and triangle-free, $y_1$ is adjacent to $v_2$ and $v_4$, and $v_3$ is not adjacent to a vertex in $X$.  Therefore $x_1$ and $x_2$ have no neighbors in $V(D)$, so $|E(V(D), X)| \leq 5$, a contradiction.

Hence, we may assume $u$ has no neighbor in $V(D)$.  By Lemma~\ref{edges leaving cube}, at least two of the vertices $\{x_1, y_1, x_2, y_2\}$ have two neighbors in $V(D)$.  Suppose $x_1$ has two neighbors in $V(D)$.  Then $y_1$ and $y_2$ have at most one, since $x_1$ does not have a common neighbor with $y_1$ or $y_2$.  Therefore $x_2$ has two neighbors in $V(D)$.  Then $y_1$ and $y_2$ have no neighbors in $V(D)$, contradicting Lemma~\ref{edges leaving cube}. Therefore we may assume by symmetry that $y_1$ and $y_2$ have two neighbors in $V(D)$.  Then $x_1$ and $x_2$ have no neighbors in $V(D)$, again contradicting Lemma~\ref{edges leaving cube}.
\end{proof}


\subsection{Discharging}

In this section, we use discharging to prove the following.

\begin{lemma}\label{finding a special vertex}
  Every triangle-free plane graph with minimum degree three contains a vertex of degree three that is not contained in a separating cycle of length four or five.
\end{lemma}

For the remainder of this subsection, suppose $G$ is a counterexample to Lemma \ref{finding a special vertex}.  We assume $G$ is connected, or else we consider a component of $G$.
Since $G$ is planar and triangle-free, it contains a vertex of degree at most three, and thus $G$ contains a separating cycle of length at most five.  We choose a separating cycle $C$ of length at most five in $G$ so that the interior of $C$ contains the minimum number of vertices, and we let $H$ be the subgraph of $G$ induced by the vertices in $C$ and its interior.  Note that $C$ has no chords since $G$ is triangle-free.  By the choice of $C$, we have the following.

\begin{claim}\label{four-fifths-sep-cycle}
  The only separating cycle of $G$ of length at most five belonging to $H$ is $C$.
\end{claim}

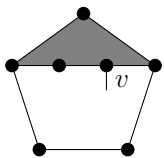
\begin{figure}
  \centering
  \begin{tikzpicture}
  \tikzstyle{every node}=[draw, fill, circle, scale = .5];
  \node (1) at (90 : 1) {};
  \node (2) at (90+72*1: 1) {};
  \node (3) at (90+72*2: 1) {};
  \node (4) at (90+72*3: 1) {};
  \node (5) at (90+72*4: 1) {};
  \draw (1) -- (2) -- (3) -- (4) -- (5) -- (1);
  \node (u) at ($(2)!.33!(5)$) {};
  \node[label=-45:{\huge $v$}] (v) at ($(2)!.66!(5)$) {};
  \draw (2) -- (u) -- (v) -- (5);
  \draw (v) -- ($(v) + (-90:.33)$);

  \begin{pgfonlayer}{bg}
      \fill[gray] (1.center) -- (2.center) -- (5.center);
  \end{pgfonlayer}
\end{tikzpicture}
  \caption{A vertex $v\in V(H)\setminus V(C)$ of degree three.}
  \label{3-vertex not in outer cycle figure}
\end{figure}

Now we need the following claim about vertices of degree three in the interior of $H$ (see Figure \ref{3-vertex not in outer cycle figure}).
\begin{claim}\label{3-vertex not in outer cycle}
  If some vertex $v\in V(H)\setminus V(C)$ has degree three, then $|V(C)| = 5$, and $v$ has precisely one neighbor in $V(C)$ and is incident to a face of length five whose boundary intersects $C$ in a subpath with three vertices.
\end{claim}
\begin{proof}
  Suppose $v\in V(H)\setminus V(C)$ has degree three.  Since $G$ is a counterexample, $v$ is contained in a separating cycle $C'$ in $G$ of length four or five.  By Claim \ref{four-fifths-sep-cycle}, $C'$ is not contained in $H$, and since $C$ is chordless, $C'$ contains a vertex not in $V(H)$.  Since $C'$ has length at most five, $v$ has at least one neighbor in $V(C)$.  By Claim \ref{four-fifths-sep-cycle}, $v$ has at most one neighbor in $V(C)$.  Hence $v$ has precisely one neighbor in $V(C)$, as desired.  Note that $V(C)\cap V(C')$ is a pair of nonadjacent vertices, or else $G$  contains a triangle.  If $v$ is not incident to a face of length five containing three vertices of $C$, or if $|V(C)| = 4$, then $H$ contains a separating cycle of length at most five containing $v$, contradicting Claim \ref{four-fifths-sep-cycle}.
\end{proof}

\begin{proof}[Proof of Lemma \ref{finding a special vertex}]
  For each $v\in V(H)\setminus V(C)$, let $\initch(v) = \deg(v) - 4$, for each $v\in V(C)$, let $\initch(v) = \deg(v) - 2$, and for each face $f$, let $\initch(f) = |f| - 4$.  Note that by Euler's formula, if $F(H)$ denotes the set of faces of $H$,
  \begin{equation*}
    \sum_{v\in V(H)}\initch(v) + \sum_{f\in F(H)}\initch(f) = 4(|E(H)| - |V(H)| - |F(H)|) + 2|V(C)| = -8 + 2|V(C)|.
  \end{equation*}
  Now we redistribute the charges in the following way, and we denote the final charge $\finalch$.  For each $v\in V(C)$, if $u\in V(H)\setminus V(C)$ has degree three and is adjacent to $v$, let $v$ send one unit of charge to $u$.
  Note that by Claim \ref{3-vertex not in outer cycle}, for each $v\in V(H)$, $\finalch(v) \geq 0$.  Note also that for each $f\in F(H)$, $\finalch(f) \geq 0$.  The sum of charges is unchanged, i.e., it is $-8 + 2|V(C)|$.

First, suppose $|V(C)| = 4$, and thus the sum of the charges is $0$.  Note that every vertex and face has precisely zero final charge, so every face has length precisely four.  By Claim \ref{3-vertex not in outer cycle}, every vertex $v\in V(H)\setminus V(C)$ has degree precisely four.  Therefore every vertex in $C$ has degree precisely two.  Since $C$ is separating, $G$ is not connected, a contradiction.

Therefore we may assume $|V(C)| = 5$, so the sum of the charges is 2.  Note that the outer face $f$ has final charge $\finalch(f) = 1$.
Since $G$ has an even number of odd-length faces, it follows that $G$ has another face $f'$ of length $5$ and final charge $1$,
and all other faces and vertices have zero final charge.  In particular, all faces of $H$ distinct from $f$ and $f'$ have length $4$
and each vertex in $V(C)$ is adjacent only to vertices of degree three in $V(H)\setminus V(C)$.
Using Claim \ref{3-vertex not in outer cycle}, we conclude there are more than two vertices of $V(H)\setminus V(C)$ with a neighbor in $C$
and at least two faces of length at least five in the interior of $C$, a contradiction.
\end{proof}

Now the proof of Theorem \ref{4/5 bound} follows easily from Lemmas \ref{min degree three}, \ref{no special vertices}, and \ref{finding a special vertex}.



\section{Proof of Theorem \ref{real-bigO-thm}}\label{bounded-degree-three-section}

For the remainder of this section, let $G$ be a counterexample to Theorem \ref{real-bigO-thm} such that $\threefaces(G)$ is minimum, and subject to that, $|V(G)|$ is minimum, and let $F$ be a set of $\threefaces(G)$ faces of $G$ such that every vertex in $G$ of degree at most three is incident to at least one of them.

\subsection{Preliminaries}
\begin{lemma}\label{bigO-min-degree-three}
  The graph $G$ has minimum degree three.
\end{lemma}
\begin{proof}
  Suppose not.  Since $G$ is planar and triangle-free, $G$ has minimum degree at most three.  Therefore we may assume $G$ contains a vertex $v$ of degree at most two.  By assumption, there is a face in $F$ incident with $v$.  Therefore $\threefaces(G - v) \leq \threefaces(G)$.  Note that $G-v$ is not 2-degenerate or else $G$ is.  By the minimality of $G$, there exists $S\subseteq V(G-v)$ of size at least $\frac{7}{8}(|V(G)| - 1) - \bigOconstant\left(\threefaces(G) - 2\right)$ such that $G[S]$ is 2-degenerate.  Now $S\cup\{v\}$ induces a 2-degenerate subgraph of $G$ on at least $\frac{7}{8}|V(G)| - \bigOconstant\left(\threefaces(G) - 2\right)$ vertices, contradicting that $G$ is a counterexample.
\end{proof}

\begin{lemma}\label{bigO-one-three-face}
  If $H$ is a triangle-free plane graph of minimum degree at least two such that $\threefaces(H) = 1$, then $H$ has at least four vertices of degree two.
\end{lemma}
\begin{proof}
  Let $f'$ be a face of $H$ incident to all the vertices in $H$ of degree at most three.  We use a simple discharging argument.  For each vertex $v$, assign initial charge $\initch(v) = \deg(v) - 4$, and for each face $f$, assign initial charge $\initch(f) = |f| - 4$.  Now let $f'$ send one unit of charge to each vertex $v$ of degree at most three incident with $f'$, and denote the final charge $\finalch$.  By Euler's formula, the sum of the charges is $-8$.  However, $\finalch(f') \geq -4$, and every other face has nonnegative final charge.  Therefore the vertices have total final charge at most $-4$.  Every vertex of degree at least three has nonnegative final charge, and every vertex $v$ of degree two has final charge $-1$.  Therefore $H$ contains at least four vertices of degree two, as desired.
\end{proof}

Lemmas~\ref{bigO-min-degree-three} and \ref{bigO-one-three-face} imply that $\threefaces(G) > 1$.
A \emph{cylindrical grid} is the Cartesian product of a path and a cycle.

\begin{lemma}\label{bigO-two-three-faces}
  If $H$ is a triangle-free plane graph such that $\threefaces(H) = 2$, then either $H$ has minimum degree at most two, or $H$ is a cylindrical grid.
\end{lemma}
\begin{proof}
Let $H$ be a triangle-free plane graph of minimum degree three such that $\threefaces(H) = 2$.  It suffices to show that $H$ is a cylindrical grid.  Let $f_1$ and $f_2$ be faces of $H$ such that every vertex of degree at most three is incident to either $f_1$ or $f_2$.  Again we use a simple discharging argument.  For each vertex $v$, assign initial charge $\initch(v) = \deg(v) - 4$, and for each face $f$, assign initial charge $\initch(f) = |f| - 4$.  Now for $i\in\{1, 2\}$, let $f_i$ send one unit of charge to each vertex $v$ incident to $f_i$, and denote the final charge $\finalch$.  By Euler's formula, the sum of the charges is $-8$.  However, $\finalch(f_1), \finalch(f_2)\geq -4$, and every other face and every vertex has nonnegative final charge.  It follows that $\finalch(f_1) = \finalch(f_2) = -4$, and that every other face and every vertex has precisely zero final charge.  Therefore the boundaries of $f_1$ and $f_2$ are disjoint, and every vertex incident with either $f_1$ or $f_2$ has degree three.  Every other vertex has degree four, and every face that is not $f_1$ or $f_2$ has length four.  It is easy to see that the only graphs with these properties are cylindrical grids, as desired.
\end{proof}

\begin{lemma}\label{bigO-cylindrical-grid}
  A triangle-free cylindrical grid on $n$ vertices contains an induced 2-degenerate subgraph on at least $\frac{7}{8}n$ vertices.
\end{lemma}
\begin{proof}
  Let $H$ be a triangle-free cylindrical grid on $n$ vertices.  The vertices of $H$ can be partitioned into $k$ sets that induce cycles $C_1, \dots, C_k$ of equal length such that for each $i\in\{2, \dots, k-1\}$, every vertex in $C_i$ has a unique neighbor in $C_{i-1}$ and in $C_{i+1}$.  Let $X$ be any set of vertices containing precisely one vertex in $C_{2i}$ for each $i\in\{1, \dots, \lfloor k/2\rfloor\}$.  Note that $H - X$ is an induced 2-degenerate subgraph on at least $\frac{7}{8}n$ vertices, as desired.
\end{proof}

By Lemmas \ref{bigO-two-three-faces} and \ref{bigO-cylindrical-grid}, we have $\threefaces(G) > 2$.

\begin{define}
  We say a subset of the plane is \textit{$G$-normal} if it intersects $G$ only in vertices.  If $f$ and $f'$ are faces of $G$, we define $d(f,f')$
to be the smallest number of vertices contained in a $G$-normal curve with one end in $f$ and the other end in $f'$.  If $f$ is a face of $G$ and $v$ is a vertex of $G$, we define $d(f,v)$ to be the minimum of $d(f,f')$ over all faces $f'$ incident with $v$.
\end{define}

\begin{lemma}\label{bigO-no-small-separator}
  Let $P$ be a $G$-normal connected subset of the plane that intersects a face in $F$ or its boundary.
  Let $X$ be the set of vertices of $G$ contained in $P$.  Suppose that $H_1$ and $H_2$ are disjoint induced subgraphs of $G - X$ such that $G - X = H_1\cup H_2$.  If $\threefaces(H_1)\ge 2$ and $\threefaces(H_2) \geq 2$, then $|X| \geq 21$.
\end{lemma}
\begin{proof}
 Note that there is a face of $G - X$ containing $P$ in its interior, and any vertex of $G - X$ of degree at most three that has degree at least four in $G$ is incident with this face.  Therefore $\threefaces(H_1) + \threefaces(H_2) \leq \threefaces(G) + 1$.  By the minimality of $G$, for each $i\in\{1, 2\}$, there exists $S_i\subseteq V(H_i)$ of size at least $\frac{7}{8}|V(H_i)| - \bigOconstant(\threefaces(H_i) - 2)$ such that $G[S_i]$ is 2-degenerate.  But $G[S_1\cup S_2]$ is 2-degenerate, and
  \begin{align*}
    |S_1\cup S_2| &\geq \frac{7}{8}\left(|V(G)| - |X|\right) - \bigOconstant(\threefaces(H_1) + \threefaces(H_2) - 4)\\
                  &\geq \frac{7}{8}|V(G)| - \bigOconstant(\threefaces(G) - 2) - \frac{7}{8}|X| + \bigOconstant.
  \end{align*}
  Since $G$ is a counterexample, $\frac{7}{8}|X| > \bigOconstant$, so $|X| \geq 21$, as desired.
\end{proof}

Note that Lemma \ref{bigO-no-small-separator} together with Lemmas~\ref{bigO-min-degree-three} and \ref{bigO-one-three-face} imply that $G$ is connected.

\begin{lemma}\label{bigO-three-faces-far-apart}
  All distinct faces $f, f'\in F$ satisfy $d(f,f')\geq 21$.
\end{lemma}
\begin{proof}
  Suppose not.  Then there is a set $X$ of at most $20$ vertices such that $f$ and $f'$ are contained in the same face of $G - X$.  Therefore $\threefaces(G - X) \leq \threefaces(G) - 1$.  

Let $n=|V(G)|$.  Recall that $\threefaces(G)\ge 3$, and thus $n>20$, as otherwise the empty subgraph satisfies the requirements of Theorem~\ref{real-bigO-thm}.
Note that $G - X$ is not 2-degenerate or else $G - X$ is an induced 2-degenerate subgraph on at least $n - 20 \geq \frac{7}{8}n - \bigOconstant(\threefaces(G) - 2)$ vertices, contradicting that $G$ is a counterexample.  So by the minimality of $G$, there exists $S\subseteq V(G-X)$ of size at least $\frac{7}{8}(|V(G)| - |X|) - \bigOconstant\left(\threefaces(G - X) - 2\right) \geq \frac{7}{8}|V(G)| - \bigOconstant\left(\threefaces(G) - 2\right)$ such that $G[S]$ is 2-degenerate, contradicting that $G$ is a counterexample.
\end{proof}

\begin{lemma}\label{layering-lemma}
  For each $f\in F$ and $k\in\{0, \dots, 9\}$, if $C_k=\{v\in V(G):d(f,v)=k\}$, then $C_k$ induces a cycle in $G$.  Furthermore, every vertex in $C_k$ has at most one neighbor $u$ satisfying $d(f,u)<k$.
\end{lemma}
\begin{proof}
  We assume without loss of generality that $f$ is the outer face of $G$.  We use induction on $k$.  In the base case, $C_0$ is the set of vertices incident with $f$.  We prove this case as a special case of the inductive step.

  By induction, we assume that for each $k' < k$, $C_{k'}$ induces a cycle in $G$ and each vertex $C_{k'}$ has at most one neighbor in $C_{k'-1}$.  Let $H = G - \bigcup_{k'=0}^{k-1} C_{k'}$.  Note that $C_k$ is the set of vertices incident with the outer face of $H$.  By Lemma \ref{bigO-three-faces-far-apart}, if $k>0$ then every vertex of $C_k$ has degree at least four in $G$.

  First we show that every $v\in C_k$ has at most one neighbor in $C_{k-1}$.  Here the base case is trivial, so we may assume $k > 0$.  Suppose for a contradiction that a vertex $v\in C_k$ has two neighbors $v_1$ and $v_2$ in $C_{k - 1}$.  Let $P_1$ and $P_2$ be the two paths in the cycle $G[C_{k-1}]$ with ends $v_1$ and $v_2$.  Since $G$ is triangle-free, $P_1$ and $P_2$ have length at least two.
For $i\in \{1,2\}$, note that the subgraph of $G$ drawn in the closure of the interior of the cycle $P_i + v_1vv_2$ has minimum degree at least two and at most three vertices ($v_1$, $v$, and $v_2$) of degree two.  Therefore by Lemma \ref{bigO-one-three-face}, it contains a face $f_i\in F$.

For $i\in\{1,2\}$, there exists a simple $G$-normal curve $A_i$ from $v_i$ to $f$ containing exactly one vertex from $C_{k'}$ for each $k'<k$.
Let $X$ consist of the vertices on $A_1$ and $A_2$ together with $v$, and note that $|X|\le 19$.  Let $G-X=H_1\cup H_2$, where $f_1$ is a face of $H_1$
and $f_2$ is a face of $H_2$---neither $f_1$ nor $f_2$ is incident with a vertex of $X$ by Lemma~\ref{bigO-three-faces-far-apart}, and for the
same reason the vertices in $H_i$ incident with $f_i$ have degree at least three for $i\in\{1,2\}$.
By Lemma~\ref{bigO-one-three-face}, for $i\in\{1, 2\}$ we have either $\threefaces(H_i)\ge 2$ or $H_i$ contains vertices of degree at most two.  In the latter case, the vertices of degree at most two in $H_i$ are incident with the outer face, and thus $\threefaces(H_i) \geq 2$.  This contradicts Lemma \ref{bigO-no-small-separator}.
Therefore every vertex of $C_k$ has at most one neighbor in $C_{k-1}$, as claimed.  Note that this implies every vertex of $C_k$ has degree at least three in $H$.
  
  Now we claim that $H$ is connected and $C_k$ does not contain a cut-vertex of $H$.  Suppose not.  Then $H$ contains at least two end-blocks $B_1$ and $B_2$.  Note that $B_1$ and $B_2$ have minimum degree at least two and at most one vertex of degree two.  Therefore by Lemma \ref{bigO-one-three-face}, $\threefaces(B_1), \threefaces(B_2) \geq 2$.  But there is a connected $G$-normal subset of the plane intersecting $G$ in a set of vertices $X$ containing only one vertex of $H$ and at most two vertices from each $C_{k'}$ for $k' < k$ such that $B_1 - X$ and $B_2 - X$ are in different components of $G - X$.  Note that $|X|\le 19$.  By Lemma \ref{bigO-three-faces-far-apart}, $\threefaces(B_1 - X), \threefaces(B_2 - X) \geq 2$, contradicting Lemma \ref{bigO-no-small-separator}.  Hence $H$ is connected, and $C_k$ does not contain a cut-vertex of $H$, as claimed.

  Since $C_k$ does not contain a cut-vertex of $H$, the outer face of $H$ is bounded by a cycle, say $C$.  Now if $C_k$ does not induce a cycle in $G$, then there is a chord of $C$, say $uv$.  Let $P_1$ and $P_2$ be paths in $C$ with ends at $u$ and $v$ such that $C = P_1\cup P_2$.  For $i\in\{1,2\}$, let $H_i$ be the graph induced by $G$ on the vertices in $P_i\cup uv$ and its interior.  Since $H_i$ has minimum degree two and at most two vertices of degree two, by Lemma \ref{bigO-one-three-face}, $\threefaces(H_i)\geq 2$.  But there is a connected $G$-normal subset of the plane containing $u$, $v$, and intersecting $G$ in a set of vertices $X$ containing at most two vertices from each $C_{k'}$ for $k' \leq k$.  Note that $|X| \leq 20$.  By Lemma \ref{bigO-three-faces-far-apart}, $\threefaces(H_1 - X), \threefaces(H_2 - X) \geq 2$, contradicting Lemma \ref{bigO-no-small-separator}.
\end{proof}

Consider a face $f\in F$ and for $k\in\{0, \dots, 9\}$, let $C_k$ be the cycle induced by $\{v\in V(G):d(f,v)=k\}$ according to Lemma~\ref{layering-lemma}.
For $k\in\{0,\ldots,8\}$ and $v\in V(C_k)$, let $n(v)$ denote the number of neighbors of $v$ in $C_{k+1}$ (note that $n(v)\ge 1$)
and $n(f,k)=\sum_{v\in V(C_k)} (n(v)-1)$.  Let $g(f,k)$ be the sum of $|f'|-4$ over all faces $f'$ such that $d(f,f')=k+1$, i.e., the faces
between cycles $C_k$ and $C_{k+1}$.  Let $b_k=3$ if $k=0$ and $b_k=4$ otherwise,
and let $c(f,k)=\sum_{v\in V(C_k)} (\deg(v)-b_k)$.  Let us also define $n(f,-1)=g(f,-1)=0$.
Observe that
$$|C_{k+1}|=|C_k|+2n(f,k)+g(f,k),$$
and 
$$n(f,k)=n(f,k-1)+g(f,k-1)+c(f,k).$$
Consequently,
$$n(f,k)=\sum_{k'=0}^k c(f,k')+\sum_{k'=0}^{k-1} g(f,k').$$

The following lemma will be crucial.
\begin{lemma}\label{lemma-degeneq}
  For every $f\in F$,
  \begin{equation*}
    8\sum_{k=0}^{9} n(f,k)\geq 249.
  \end{equation*}
\end{lemma}

First we need the following claims.
\begin{claim}\label{lemma-basiclayers}
  Every face $f\in F$ satisfies $n(f, 1)\ge 2$.
\end{claim}
\begin{proof}
  Suppose that $n(f,1)\le 1$.
  For $k\in \{0,1\}$, let $C_k$ denote the cycle induced by $\{v\in V(G):d(f,v)=k\}$ according to Lemma~\ref{layering-lemma}.
  If $n(f,1)=0$, then $c(f,0)=c(f,1)=0$ and $g(f,0)=0$, i.e., $H=G[V(C_0\cup C_1)]$ is a cylindrical grid
  and all vertices of $C_1$ have degree $4$ in $G$.  In this case, let $v$ be an arbitrary vertex of $C_1$.
  If $n(f,1)=1$, then $c(f,0)+g(f,0)+c(f,1)=1$, so one of the following holds (see Figure \ref{one charge in first layer fig}):
  \begin{itemize}
  \item $c(f, 0) = 1$, so there is a vertex $v'\in V(C_0)$ of degree four and a vertex $v''\in V(C_1)$ of degree two in $H$; we let $v$ be any vertex of $C_1$ that is not $v''$ and is not adjacent to $v'$.  Note that every vertex of $V(C_0)\setminus\{v'\}$ has degree three, and every vertex of $C_1$ has degree four in $G$.  Or,
  \item $g(f, 0) = 1$, so there is a face of $H$ of length five incident with a vertex $v'\in V(C_1)$ of degree two in $H$; we let $v$ be any vertex of $C_1$ other than $v'$.  Note that every vertex of $C_0$ has degree three and every vertex of $C_1$ has degree four in $G$.  Or,
  \item $c(f, 1) = 1$, so $H$ is a cylindrical grid and exactly one vertex of $C_1$ has degree five; we let $v$ be this vertex.
  \end{itemize}
  
  Let $X=V(C_1\cup C_2)$.
  Note that $\threefaces(G-X)\le \threefaces(G)$,
  and by the minimality of $G$, there exists $S\subseteq V(G-X)$ inducing a $2$-degenerate subgraph such that
  $|S|\ge \frac{7}{8}|V(G-X)| - \bigOconstant\left(\threefaces(G-X) - 2\right)\ge \frac{7}{8}|V(G)| - \bigOconstant\left(\threefaces(G) - 2\right)-(|X|-1)$.
  But then $S\cup (X\setminus\{v\})$ induces a $2$-degenerate subgraph of $G$, contradicting the assumption that $G$ is a counterexample.
\end{proof}

\begin{figure}
  \centering
  \begin{minipage}[b]{.33\linewidth}
  \centering
  \begin{tikzpicture}
    \node[draw=none, fill=none] at (0, 0) {$f$};
    \tikzstyle{every node}=[draw, fill, circle, scale = .5];
    \node[label=135:{\huge$v'$}] (1) at (135 : 1) {};
    \node (2) at (-135 : 1) {};
    \node (2') at (-135 : 2) {};
    \draw (2) -- (2') -- ($(2') + (-135 : .5)$);
    \node (3) at (-45 : 1) {};
    \node (3') at (-45 : 2) {};
    \draw (3) -- (3') -- ($(3') + (-45 : .5)$);
    \node (4) at (45 : 1) {};
    \node[label=90:{\huge$v$}] (4') at (45 : 2) {};
    \draw (4) -- (4') -- ($(4') + (45 : .5)$);
    \draw (1) -- (2) -- (3) -- (4) -- (1);
    \node (1a) at (2' |- 1) {};
    \node (1b) at (1 |- 4') {};
    \draw (1) -- (1a) -- ($(1a) + (180: .5)$); \draw (1) -- (1b) -- ($(1b) + (90: .5)$);
    \node[label=135:{\huge$v''$}] (1') at (1a |- 1b) {};
    \draw ($(1') + (90 : .5)$) -- (1') -- ($(1') + (180: .5)$);
    
    \draw (1') -- (1a) -- (2') -- (3') -- (4') -- (1b) -- (1');
  \end{tikzpicture}\\$c(f, 0) = 1$%
\end{minipage}%
\begin{minipage}[b]{.33\linewidth}
  \centering
  \begin{tikzpicture}
    \node[draw=none, fill=none] at (0, 0) {$f$};
    \tikzstyle{every node}=[draw, fill, circle, scale = .5];
    \node (1) at (135 : 1) {};
    \node (2) at (-135 : 1) {};
    \node (2') at (-135 : 2) {};
    \draw (2) -- (2') -- ($(2') + (-135 : .5)$);
    \node (3) at (-45 : 1) {};
    \node (3') at (-45 : 2) {};
    \draw (3) -- (3') -- ($(3') + (-45 : .5)$);
    \node (4) at (45 : 1) {};
    \node[label=90:{\huge$v$}] (4') at (45 : 2) {};
    \draw (4) -- (4') -- ($(4') + (45 : .5)$);
    \draw (1) -- (2) -- (3) -- (4) -- (1);
    \node (1') at (135 : 2) {};
    \draw (1) -- (1') -- ($(1') + (135 : .5)$) {};
    \node[label=90:{\huge$v'$}] (v') at ($(1') !.5! (4')$) {};
    \draw ($(v') + (135 : .5)$) -- (v') -- ($(v') + (45 : .5)$);
    \draw (1') -- (2') -- (3') -- (4') -- (1');
  \end{tikzpicture}\\$g(f, 0) = 1$%
\end{minipage}%
\begin{minipage}[b]{.33\linewidth}
  \centering\begin{tikzpicture}
    \node[draw=none, fill=none] at (0, 0) {$f$};
    \tikzstyle{every node}=[draw, fill, circle, scale = .5];
    \node (1) at (135 : 1) {};
    \node (2) at (-135 : 1) {};
    \node (2') at (-135 : 2) {};
    \draw (2) -- (2') -- ($(2') + (-135 : .5)$);
    \node (3) at (-45 : 1) {};
    \node (3') at (-45 : 2) {};
    \draw (3) -- (3') -- ($(3') + (-45 : .5)$);
    \node (4) at (45 : 1) {};
    \node[label=45:{\huge$v$}] (4') at (45 : 2) {};
    \draw (4) -- (4');
    \draw ($(4') + (90 : .5)$) -- (4') -- ($(4') + (0 : .5)$);
    \draw (1) -- (2) -- (3) -- (4) -- (1);
    \node (1') at (135 : 2) {};
    \draw (1) -- (1') -- ($(1') + (135 : .5)$) {};
    \draw (1') -- (2') -- (3') -- (4') -- (1');
  \end{tikzpicture}\\$c(f, 1) = 1$%
\end{minipage}
  \caption{$n(f, 1) = 1$, when $|f| = 4$.}
  \label{one charge in first layer fig}
\end{figure}
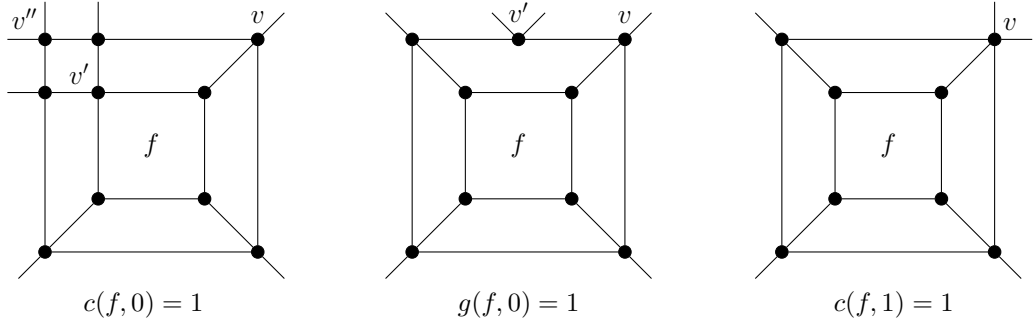

\begin{claim}\label{layer size lower bound}
  Let $f\in F$ and for $k\in\{0, \dots, 9\}$, let $C_k$ be the cycle induced by $\{v\in V(G) : d(f, v) = k\}$ according to Lemma~\ref{layering-lemma}.  Then
  \begin{equation*}
  |C_0\cup \ldots\cup C_{9}|\ge 184.
\end{equation*}
\end{claim}
\begin{proof}
  Claim~\ref{lemma-basiclayers} implies that $n(f,k)\ge 2$ for $k\in\{1,\ldots,9\}$,
  and thus $|C_{k+1}|\ge |C_k|+4$ for $k\in \{1,\ldots,8\}$.  Since $G$ is triangle-free, we have $|C_0|\ge 4$, and
  we conclude that
  \begin{equation*}
    |C_0\cup\ldots\cup C_{9}| \geq   10|C_0|+4(1+2+\ldots+8) \geq 184,
  \end{equation*}
  as desired.
\end{proof}

Now we can prove Lemma~\ref{lemma-degeneq}.
\begin{proof}[Proof of Lemma \ref{lemma-degeneq}]
  For $k\in\{0,\ldots,9\}$, let $C_k$ denote the cycle induced by $\{v\in V(G):d(f,v)=k\}$ according to Lemma~\ref{layering-lemma}.
Let $X=\bigcup_{k=0}^{9} V(C_k)$.  By Claim~\ref{layer size lower bound}, $|X| \geq 184$.
For $k\ge 1$, let $R_k$ be a smallest subset of $V(C_k)$ such that $\sum_{v\in V(C_k)\setminus R_k} (n(v) - 1)\le 1$.
By Claim~\ref{lemma-basiclayers} and the monotonicity of $n(f,k)$, we have $n(f,k)\ge 2$, and thus $R_k$ is non-empty.
Note that $|R_k|\le n(f,k)-1$.  For $k=0$, let $R_0$ be defined in the same way if $n(f,0)\ge 2$, and let $R_0$ consist
of an arbitrary vertex of $C_0$ otherwise.  Let $R=\bigcup_{k=0}^{9} R_k$.

By Lemma~\ref{bigO-three-faces-far-apart}, we have $\threefaces(G-X)\le\threefaces(G)$, and by the minimality of $G$,
there exists $S\subseteq V(G-X)$ inducing a $2$-degenerate subgraph such that
$|S|\ge \frac{7}{8}|V(G-X)| - \bigOconstant\left(\threefaces(G-X) - 2\right)\ge \frac{7}{8}|V(G)| - \bigOconstant\left(\threefaces(G) - 2\right)-\frac{7}{8}|X|$.
We claim that $S\cup (X\setminus R)$ induces a $2$-degenerate subgraph of $G$.  Indeed, it suffices to show that for every non-empty
$X'\subseteq X\setminus R$, the graph $G[S\cup X']$ has a vertex of degree two.  Let $k$ be the minimum index such that $X'\cap V(C_k)\neq\emptyset$.
Note that by the choice of $R_k$, $C_k[X']-R_k$ is a union of paths containing at most one vertex with more than one neighbor in $C_{k+1}$,
and if there is such a vertex, it has exactly two neighbors in $C_{k+1}$.
Consequently, one of the endvertices of these paths
has degree at most two in $G[S\cup X']$.

Since $G$ is a counterexample, we conclude that $|X\setminus R|<\frac{7}{8}|X|$, and thus $8|R|-1 \geq |X| \geq 184$.
Since $|R|\le 2+\sum_{k=0}^{9} (n(f,k)-1) \leq -8 + \sum_{k=0}^{9}n(f, k)$, the inequality
$$
8\sum_{k=0}^{9}n(f, k)\ge 249
$$ follows.
\end{proof}


\subsection{Discharging}

In this subsection we use discharging to complete the proof of Theorem \ref{real-bigO-thm}.
\begin{proof}[Proof of Theorem \ref{real-bigO-thm}]
For each $v\in V(G)$, let $\initch (v) = \deg(v) - 4$, and for each face $f$ of $G$, let $\initch(f) = |f| - 4$.  Now we redistribute the charges according to the following rules and denote the final charge by $\finalch$.

\begin{enumerate}
\item Every face $f\in F$ sends 1 unit of charge to every vertex incident with $f$.
\item Afterwards, every face $f' \notin F$ and every vertex $v\in V(G)$ such that $d(f,f')\le 9$ or $d(f,v)\le 9$
for some face $f\in F$ sends all of its charge to $f$.
\end{enumerate}

Observe that every vertex and every face not in $F$ sends its charge to at most one face of $F$ by Lemma \ref{bigO-three-faces-far-apart}.
Clearly, all vertices and all faces not in $F$ have non-negative final charge.  By Euler's formula the sum of the charges is $-8$, so there exists some face $f\in F$ with negative charge.

By Lemma \ref{layering-lemma}, for each $k\in\{0, \dots, 9\}$, the vertices $v\in V(G)$ such that $d(f,v)=k$ induce a cycle in $G$, say $C_k$.
Note that after the first discharging rule is applied, $f$ has charge $-4$, and since $\finalch(f)\le -1$, at most three units of charge
are sent to $f$ according to the second rule.  Note that $f$ receives precisely $c(f, k)$ total charge from vertices of $C_k$ and
precisely $g(f, k)$ total charge from faces between $C_k$ and $C_{k+1}$.  Hence, we have
$$3\ge \sum_{k'=0}^{9} c(f,k')+\sum_{k'=0}^{8} g(f,k')\ge n(f,k)$$
for every $k\in\{0,\ldots,9\}$.  Therefore,
\begin{equation*}
  \sum_{k=0}^{9} n(f,k)\le 30.
\end{equation*}
However, this contradicts Lemma~\ref{lemma-degeneq}, finishing the proof.
\end{proof}

Let us remark that the constant $18$ in the statement of Theorem~\ref{real-bigO-thm} can be improved.  In particular, one could extend the case
analysis of Claim~\ref{lemma-basiclayers} to fully describe larger neighborhoods of the face, likely obtaining enough charge 
in a much smaller number of layers than $10$ needed in our argument (at the expense of making the proof somewhat longer and harder to read).



\bibliographystyle{plain}
\bibliography{planar-degen-subgraphs}

\end{document}